%% file: article.tex
\documentclass[a4paper,11pt]{amsart} 
\input makra
\begin{document}

\title[Generalized Dolbeault sequences in parabolic geometry]{Generalized Dolbeault sequences in parabolic geometry}
\author{Peter Franek}

\begin{abstract}
In this paper, we show the existence of a sequence of invariant differential operators 
on a particular homogeneous model $G/P$ of a Cartan geometry. 
The first operator in this sequence can be locally identified with the Dirac operator 
in $k$ Clifford variables, $D=(D_1,\ldots, D_k)$, 
where $D_i=\sum_j e_j\cdot \partial_{ij}: C^\infty((\R^n)^k,\S)\to C^\infty((\R^n)^k,\S)$. 
We describe the structure of these sequences in case the dimension $n$ is odd.
It follows from the construction that all these operators are invariant 
with respect to the action of the group $G$. 

These results are obtained by constructing homomorphisms 
of generalized Verma modules, what are purely algebraic objects.
\end{abstract}

\maketitle

\section{Motivation}

There are two basic generalizations of the space of holomorphic functions to higher dimensions. 
One of them is the notion of holomorphic functions in several variables, 
$f: {\R}^{2k}\simeq \C^k\to \C$, $\bar{\partial_j}f=0$
for $j=1,\ldots,k$. 
The second possible generalization deals with s.c. {\it monogenic functions}, 
which are defined on $\R^n$ with values in the
{\it Clifford algebra} or the {\it space of spinors} and solve the {\it Dirac equation}
$\sum_j e_j\cdot \partial_j f=0$. 
They have similar nice properties as holomorphic functions 
and coincide with them for $n=2$ (\cite{GM}). 

Recently, many variations and generalizations of the classical Dirac operator appeared. 
While mathematical physicists study its spectra on different Riemannian spin-manifolds 
and other construct its analogs in non-riemannian geometries (see e.g. \cite{svata}), 
we may define the {\it Dirac operator in several Clifford variables} by
$D: C^\infty((\R^n)^k, \S)\to C^\infty((\R^n)^k, \C^k\otimes\S)$, 
$D=(D_1,\ldots,D_k)$ (after identifying elements of the image with $k$ spinor valued functions), 
$D_i=\sum_j e_j\cdot \partial_{ij}$ where $\S$ is the
(usually complex) spinor space, $x_{uv}$ the standard coordinates on 
$(\R^n)^k, u=1,\ldots,k, j=1,\ldots,n$, and $\cdot$ the Clifford multiplication 
$\R^n\times\S\to\S$.

This is a common generalization of the space of holomorphic functions
in several complex variables ($n=2$, $k$ arbitrary) and the classical Dirac operator ($k=1$). 


Many problems can be studied using a resolution of $D$, 
i.e. the (locally) exact complex of PDE's starting with the operator $D$.
In the case of holomorphic functions in several complex variables, $D$ being the 
Cauchy--Riemann operator ($n=2$, $k$ arbitrary), this is just the Dolbeault sequence.
For $k=2$, $n$ even, the problem was studied in \cite{Tampere, KrumpSoucek}.
However, for arbitrary $n,k$, the form of this resolution is not known yet, except of some 
special cases (see \cite{BuresDam, bluebook, InvRes, Struppa3}).

In this paper, the problem is treated in the framework of {\it parabolic geometry} and
some particular results are obtained for $n$ odd, $k$ arbitrary. We construct sequences
of differential operators starting with the Dirac operator $D$ that are good candidates
for being a resolution (the proof that they indeed form a resolution is still in progress).
Our sequences contain all operators that are {\it invariant} with respect to the action
of a quite large group and continue the Dirac operator.

Because the space of spinors arises naturally as a fundamental representation of the
Lie group $\Spin(n)$, it is natural to consider the Dirac operator as acting not 
only on $C^\infty(\R^n, \S)$ but rather on more general sections of a spinor bundle over a spin manifold $M$ 
(see \cite{Friedrich}).
The simplest spin structure on the sphere $S^n$ is the bundle 
$\Spin(n+1)\to \Spin(n+1)/\Spin(n)\simeq S^n$ and the associated spinor bundle 
is $\Spin(n+1)\times_{\Spin(n)} \S$. The usual Dirac operator acts between
sections of this bundle and is invariant with respect to the group $\Spin(n+1)$
(the sections $\Gamma(G\times_H \V)$ can be naturally identified with 
invariant functions $C^\infty(G,\V)^H$ and the action of $G$ is $g\cdot f(x):=f(g^{-1}x)$).
However, Dirac operator has a larger group of invariance. Whereas $\Spin(n+1)$
acts on the sphere by rotations, it is well known that Dirac operator is invariant 
with respect to all M\"obius transformations.
This is reflected by the fact that the bundle $\Spin(n+1)\to \Spin(n+1)/\Spin(n)$ is
a reduction of a larger bundle $\Spin(n+1,1)/P$, where $\Spin(n+1,1)$ acts on the 
null-cone of a form $g$ of signature $(n+1,1)$ that defines the group $\Spin(n+1,1)$.
The projectivisation of this
null-cone is homeomorphic to the sphere $S^n$ and $P$ is the stabilizer of one line.
It was shown in \cite{srni} that considering $\S_1$ as a representation of $P$ with 
highest weight 
$$
\dynkin \noroot{\frac{n}{2}-1}\link\whiteroot{0}\link\ldots\link\whiteroot{0}
\whiterootdownright{0}\whiterootupright{1}\enddynkin\quad\hbox{resp.}\quad
\dynkin \noroot{\frac{n}{2}-1}\link\whiteroot{0}\link\ldots\link\whiteroot{0}
\llink>\whiteroot{1}\enddynkin
$$ and $\S_2$ a representation of $P$ with highest weight 
$$
\dynkin \noroot{\frac{n}{2}}\link\whiteroot{0}\link\ldots\link\whiteroot{0}
\whiterootdownright{1}\whiterootupright{0}\enddynkin\quad\hbox{resp.}\quad
\dynkin \noroot{\frac{n}{2}}\link\whiteroot{0}\link\ldots\link\whiteroot{0}
\llink>\whiteroot{1}\enddynkin,
$$
the Dirac 
operator is a $\Spin(n+1,1)$-invariant differential operator 
$D: \Gamma(\Spin(n+1,1)\times_{P} \S_1)\to \Gamma(\Spin(n+1,1)\times_P \S_2)$.
In this sense, the Dirac operator is conformally invariant, as $\Spin(n+1,1)$
(or, more exactly, its connected component) 
is the double-cover of the group of all M\"obious transformations.

The subalgebra $P$ is a {\it parabolic subalgebra} of $G=\Spin(n+1,1)$, 
i.e. its Lie algebra $\liep$ contains a Borel algebra $\lieb$ of $\lieg$, 
the Lie algebra of $G$.  
The bundle $G\to G/P$ together with the Maurer-Cartan form on $T(G)$ is an example
of s.c.  {\it parabolic geometry} (see \cite{Cap, Sharpe}).

In \cite{srni}, an analogous construction is described for the group $G=\Spin(n+k,k)$
and $P$ being a parabolic subgroup fixing a maximal vector subspace of 
the null cone of the metric of signature $(n+k,k)$ defining $\Spin(n+k,k)$.
The reductive part of $P$ is isomorphic to
$\GL(k)\times \Spin(n)$. The Lie algebra $\liep$ of $P$ determines a gradation
of the Lie algebra $\lieg$ of $G$ so that $\lieg=\oplus_{j=-2}^2 \lieg_j$ and 
$\liep=\lieg_0\oplus \lieg_1\oplus \lieg_2$.
Again, choosing proper irreducible $P$-modules $\V_1$ resp. $\V_2$ with
highest weights 
\begin{eqnarray*}
&&\mmmm \dynkin\whiteroot{0}\link\ldots\link\whiteroot{0}\link\noroot{\frac{n}{2}-1}\link\whiteroot{0}\link\ldots\link\whiteroot{0}\llink>\whiteroot{1}\enddynkin\quad\hbox{resp.}\\
&&\mmmm \dynkin\whiteroot{1}\link\whiteroot{0}\link\ldots\link\whiteroot{0}\link\noroot{\frac{n}{2}-1}\link\whiteroot{0}\link\ldots\link\whiteroot{0}\llink>\whiteroot{1}\enddynkin
\end{eqnarray*} 
(and similar for $n$ even), we showed in \cite{srni, dizertace}
that there exists a $G$-invariant differential operator 
$D: \Gamma(G\times_P \V_1)\to \Gamma(G\times_P \V_2)$ and, identifying local sections
in the neighborhood of $eP$ with $\V_i$-valued functions on the vector space 
$\lieg_-=\oplus_{j<0} \lieg_j$ in a natural way 
and restricting to functions that are constant in $\lieg_{-2}\subset\lieg_-$,
this operator coincides with the Dirac operator in $k$ Clifford variables (identifying $\lieg_{-1}\simeq (\R^n)^k$
as the adjoint representation of $\lieg_0\simeq \gl(k)\times\so(n)$).

The question is, whether we can find sequences
of $G$-invariant differential operators continuing the operator $D$. In case of the Dirac
operator in one variable ($k=1$), this is not possible. 
We showed in \cite{Tampere} that for $k=2$, there exist two further $G$-invariant differential 
operators so that they form a complex together with the first one. 

In general, for any semisimple Lie group $G$,  a parabolic subgroup $P$ and
some $P$-modules $\V_1, \V_2$, 
the $G$-invariant differential operators between sections of vector bundles
$D: \Gamma(G\times_P \V_1)\to \Gamma(G\times_P \V_2)$ are in $1-1$ correspondence
with the $\lieg$-homomorphisms of generalized Verma modules $M_{\liep}(\V_2^*)\to M_{\liep}(\V_1^*)$
induced by dual representations $\V_2^*$ and $\V_1^*$ (see \cite{Cap}). 
Therefore, the generalized Verma modules and their homomorphisms 
will be studied in the rest of this paper.

\section{Basics on Verma modules}
\subsection{Bruhat ordering}

Let as assume that $\liep$ is a parabolic subalgebra of $\lieg$, i.e. a subalgebra containing the Borel subalgebra
$\lieb$. This induces a gradation 
$\oplus_{j=-k}^k \lieg_j$ of $\lieg$ so that $\liep=\sum_{j\geq 0} \lieg_j$. Let $\lieh$ 
be a fixed Cartan subalgebra of $\lieg$ and $\liep$,
$\Phi^+$ a set of positive roots of $\lieg$ (and also of $\liep$) and $\Delta$ the set of
simple roots, compatible with $\Phi^+$. There is a $1-1$ correspondence between subsets 
$\Sigma$ of $\Delta$ and parabolic subalgebras $\liep_\Sigma\subset\lieg$, where $\liep_\Sigma$
contains the Cartan subalgebra, all positive root spaces and all those negative root spaces $\lieg_{-\beta}$, 
such that $\beta$ can be expressed as a sum of simple roots from $\Delta-\Sigma$. These roots form the set of simple roots of the algebra $\lieg_0$ from the associated grading $\oplus_{j=-k}^k \lieg_j$.
In the Dynkin diagram, we draw the simple roots in $\Sigma$ as crossed ($\times$). 

For any pair $(\lieg, \liep)$ there exists a unique element $E\in \lieg$ called 
{\it grading element} so that $\ad(E)(X)=jX$ for any $X\in\lieg_j,\,\,\,j=-k,\ldots,k$.

For each $\beta\in\Phi^+$, the {\it root reflection} $s_\beta$ is a reflection in $\lieh^*$ fixing the
hyperplane orhogonal to $\beta$ in the Killing metric. In coordinates, $s_\beta(\gamma)=\gamma-\gamma(H_\beta)\beta$
where $H_\beta$ is the $\beta$-coroot (see e.g. \cite{Humphries}).
The choice of $\Delta$ determines 
the length $l(w)$ of any element $w$ of the Weyl group $W$ of 
$\lieg$. It is the minimal number $k$ such that $w=s_{\alpha_{i_1}}\ldots s_{\alpha_{i_k}}$, 
$\alpha_{i_j}\in \Delta$, $s_{\alpha_i}$ being simple root reflections. 
This defines the Bruhat ordering on $W$ in the following way:
$w\leq w'$ if and only if there exist $w=w_0\to w_1\to w_2\to\ldots\to w_l=w'$, where 
$w_i\to w_{i+1}$ means that $w_{i+1}=s_{\beta_i} w_i$ for some $\beta_i\in\Phi^+$ and
the length $l(w_{i+1})=l(w_{i})+1$.

\subsection{Generalized Verma modules (GVM)}

Let $\V$ be a (usually finite dimensional) irreducible $\liep$-module with
highest weight $\lambda$. 
The generalized Verma module (further GVM), introduced by 
Lepowski (\cite{Lepowski}) is defined by $M_{\liep}(\V):=\univ(\lieg)\otimes_{\univ(\liep)} \V$,
where $\univ(\lieg)$ is the universal enveloping algebra of $\lieg$, considered as a left 
$\univ(\lieg)$ and a right $\univ(\liep)$-module. $M_{\liep}(\V)$ is a highest weight module with highest weight
$\lambda$ and highest weight vector $1\otimes v_\lambda$, where $v_\lambda$ is a highest weight 
vector in $\V$.
As a $\lieg_-$-module and $\lieg_0$-module, 
$M_{\liep}(\V)\simeq \univ(\lieg_-)\otimes \V$. The GVM is uniquely determined by its 
highest weight $\lambda$, therefore we will sometimes denote the GVM with highest weight $\lambda$
by $M_{\liep}(\lambda+\delta)$, where $\delta=\half\sum_{\beta\in\Phi^+} \beta$.
Assuming that $\V$ is finite dimensional, the set of GVM's is 
isomorphic to the set of $\liep$-dominant and $\liep$-integral weights 
(this means weights $\lambda$ such that $\lambda(H_\alpha)$ is non-negative and integral for each $\alpha\in\Delta-\Sigma$).
Such weights will be further denoted by $P_{\liep}^{++}$.

If $\liep=\lieb=\lieh\oplus_{\beta\in\Phi^+}\lieg_\beta$ is the Borel subalgebra of $\lieg$, the GVM
$M_{\lieb}(\V)$ is called true Verma module, or simply Verma module 
(in this case, $\V$ is a one-dimensional representation 
of $\lieb$ and its weight can be any $\lambda\in\lieh^*)$. Each highest weight module with highest weight $\lambda$ is isomorphic to some factor of the Verma module $M_{\lieb}(\lambda+\delta)$.

\subsection{Duality between GVM homomorphisms and invariant differential operators}%
A $G$-invariant differential operator $D: \Gamma(G\times_P \V)\to \Gamma(G\times_P \W)$ 
is completely determined by the values $Ds(eP)$ on sections ($e\in G$ is the identity element). 
If the operator is of order $k$, the value $Ds(eP)$ depends only on
the $k$-jet $J^k_{eP} s$ of a section $s$ in $eP$. So, the operator $D$ is determined by a 
map $\tilde{D}: J^k_{eP}(G\times_P \V)\to \W$ that evaluates the image of a section $s$ in $eP$, identifying the fiber over $eP$ with $\W$ in a natural way. More precisely, $D(s)(eP)=[e,\tilde{D}(j^k_{eP}s))]_P$.
Because $D$ is $G$-invariant, $\tilde{D}$ has to be $P$-invariant, the action of $P$ on the jets being the action on representatives.

The $P$-module $J_{eP}^k(G\times_P \V)$ of $k$-jets of sections is naturally isomorphic to the space 
of $k$-jets of $P$-invariant functions $J_{e}^{k}(C^\infty(G,\V)^P)$ 
(the action of $P$ here being $(p\cdot f)(x)=f(p^{-1}x)$).
It can be shown that this is dual, as a $P$-module, to 
$\univ_k(\lieg)\otimes_{\univ(\liep)} \V^*$
(where $\univ_k(\lieg)$ is the $k$-th
filtration of $\univ(\lieg)$) and the duality is given by 
\begin{equation}
\label{duality}
(Y_1\ldots Y_l\otimes_{\univ(\liep)} A) (j_e^k f):=A((L_{Y_1}\ldots L_{Y_l} f)(e))
\end{equation}
for $l\leq k$, $A\in \V^*$, $j_e^k f$ the $k$-jet of $f$ in $e$, 
$Y_j\in\lieg$ and $L_{Y_j}$ the derivation with respect to the
left invariant vector fields on $G$ induced by $Y_j$ (see \cite{Cap} for details). 

Any $P$-homomorphism $\tilde{D}: J_{e}^k(C^\infty(G,\V)^P)\to\W$
is determined by its dual map 
$\tilde{D}^*: \W^*\to J_{e}^{k}(C^\infty(G,\V)^P)^*$
and we see from $(\ref{duality})$ that the right hand side can be identified with a $P$-submodule of 
$M_{\liep}(\V^*)$. Further, each $P$-homomorphism 
$\tilde{D}^*: \W^*\to M_{\liep}(\V^*)$ 
can be extended to a $(\lieg,P)$-homomorphism $M_{\liep}(\W^*)\to M_{\liep}(\V^*)$
of GVM's by 
$y_1\ldots y_l\otimes v\mapsto y_1\ldots y_l \tilde{D}^*(v)$ for $y_j\in\lieg_-$, 
the action of $\liep$ on $\W^*$ being the infinitesimal action of $P$ 
(we identified $M_{\liep}(\W^*)\simeq \univ(\lieg_-)\otimes \W^*)$.

It follows that there is a duality between invariant linear differential operators 
$D: \Gamma(G\times_P \V)\to \Gamma(G\times_P \W)$ of any finite order and $(\lieg,P)$-homomorphisms
of GVM's $M_{\liep}(\W^*)\to M_{\liep}(\V^*)$. Note that, if the inducing representations
$\V$ and $\W$ are both $P$-modules, then each $\lieg$-homomorphism $M_{\liep}(\V)\to M_{\liep}(\W)$ 
is a $(\lieg, P)$-homomorphism as well. 

Finally, note that if the Lie groups $(G,P)$ are real but the representation spaces $\V, \W$ are complex
representations of $P$, then the real GVM $M_{\liep}(\V)$ is ($\lieg$-) isomorphic to the complex GVM induced by $\V$, considered
as a complex representation of the complexified Lie algebra $\liep^{\C}$. Therefore, we may restrict to GVM's associated
to complex Lie algebras $(\lieg^{\C}, \liep^{\C})$.

\subsection{Homomorphisms of GVM's}
The GVM's are highest weight modules, therefore they admit central characters. As each
$\lieg$-homomorphism of highest weight modules must preserve the central character, it follows from 
Harris-Chandra theorem (see, e.g. \cite{Humphries}) that a $\lieg$-homomorphism \gvmhom may exist only if $\mu$ and $\lambda$
are on the same orbit of the Weyl group $W$ of the Lie algebra $\lieg$. (Recall that 
the highest weights of these modules are $\mu-\delta$ and $\lambda-\delta$.) 
For $\lambda\in P_{\liep}^{++}+\delta$, there exist only a finite number
of weights $\mu\in P_{\liep}^{++}+\delta$ on the same orbit of the Weyl grup.

In the case of true Verma modules, there is a classification of their homomorphisms, done by 
Verma and Bernstein-Gelfand-Gelfand (\cite{bgg1, bgg2, Verma}), summarized 
in the following statements:

\begin{theorem}
\label{truevermamodulesmap}
Let $\mu, \lambda\in\lieh^*$. Each homomorphism $M_{\lieb}(\mu)\to M_{\lieb}(\lambda)$
is injective and $\dim(\Hom(M_{\lieb}(\mu), M_{\lieb}(\lambda)))\leq 1$. Therefore, we can write
$M_{\lieb}(\mu)\subset M_{\lieb}(\lambda)$ in such case.

A nonzero homomorphism of Verma modules $M_{\lieb}(\mu)\to M_{\lieb}(\lambda)$ exists if and
only if there exist weights $\lambda=\lambda_0, \lambda_1,\ldots,\lambda_k=\mu$
so that $\lambda_{i+1}=s_{\beta_i}\lambda_{i}$ for some positive roots $\beta_i$
and $\lambda_i(H_{\beta_i})\in\N$ for all $i$ ($s_\beta\in W$
is the $\beta$-root reflection). Equivalently, $\lambda_i-\lambda_{i-1}$ is 
a positive integral multiple of some positive root for all $i$.

Let $\lambda\in P_{\lieg}^{++}+\delta$ (i.e. $\lambda-\delta$ is $\lieg$-dominant 
and $\lieg$-integral). Then there exists a nonzero homomorphism 
$M_{\lieb}(w'\lambda)\to M_{\lieb}(w\lambda)$ if and only if $w\leq w'$ in the Bruhat ordering.

If $\lambda$ is only $\lieg$-dominant ($\lambda(H_\beta)>0$ for all $\beta\in\Phi^+$), then
the existence of a nonzero standard homomorphism $M_{\lieb}(w'\lambda)\to M_{\lieb}(w\lambda)$ 
still implies $w\leq w'$ in the Bruhat ordering 
(but not the opposite).
\end{theorem}

Because $M_{\liep}(\lambda)$ is a highest weight module, it is isomorphic to a factor of 
true Verma module $M_{\lieb}(\lambda)/W$. It was proved by Lepowski that 
$W\simeq \sum_{\alpha\in\Delta-\Sigma}M_{\lieb}(s_\alpha\lambda)$ ($\Sigma\subset\Delta$ determines
the parabolic subalgebra $\liep$ and all the modules $M_{\lieb}(s_\alpha\lambda)$ are considered
as submodules of $M_{\lieb}(\lambda)$). A homomorphism
\gvmhom is called standard, if it is a factor of a true Verma module homomorphism 
$M_{\lieb}(\mu)\to M_{\lieb}(\lambda)$. Up to multiple, there exists at most one standard homomorphism 
from $M_{\liep}(\mu)$ to $M_{\liep}(\lambda)$. 
The following is known about standard homomorphisms of GVM's:

\begin{theorem}
\label{zeromap}
Let $\mu, \lambda\in P_{\liep}^{++}+\delta$,
$i:M_{\lieb}(\mu)\to M_{\lieb}(\lambda)$ be a homomorphism of Verma modules.
Then the standard homomorphism $M_{\liep}(\mu)\to M_{\liep}(\lambda)$ is zero if and only 
if there exists 
$\alpha\in \Delta-\Sigma$ so that $i(M_{\lieb}(\mu))\subset M_{\lieb}(s_\alpha\lambda)$ (identifying
$M_{\lieb}(s_\alpha\lambda)$ with a submodule of $M_{\lieb}(\lambda)$).
\end{theorem}

Let us denote by $W_{\liep}$ the subgroup of $W$ generated by simple root reflections 
\{$s_\alpha,\,\,\alpha\in\Delta-\Sigma\}$ and 
$W^p$ the subset of $W$ consisting of those $w\in W$ so that 
$w\tilde\lambda$ is $\liep$-dominant for each $\lieg$-dominant weight $\tilde\lambda$.
Any $w\in W$ can be uniquely decomposed $w=w_p w^p$
where $w_p\in W_p$ and $w^p\in W^p$ and the length $l(w)=l(w_p)+l(w^p)$. 
We define the {\it parabolic Hasse graph} for $(\lieg, \liep)$ to be the set $W^p$ of vertices with 
arrows $w\to w'$ if and only if $w\to w'$ in $W$. 

The following two properties of the parabolic Hasse graph will be used later
(for the proof, see \cite{Bjorner}):
\begin{lemma}
\label{path}
(1)
If $w'=s_\gamma w$, then either $w\leq w'$ or $w'\leq w$ in the Bruhat ordering.

(2)
Let $w, w'\in W^p$ and $w\leq w'$ in the Bruhat ordering. Then there exists a path 
$w\to w_1\to \ldots \to w_n\to w'$ so that all $w_i$ are in $W^p$.
\end{lemma}

The following theorem can be used to prove the existence of a standard GVM homomorphism:

\begin{theorem}
\label{genlepowski}
Let $\tilde{\lambda}$ be a strictly dominant weight (i.e. $\tilde\lambda(H_\beta)>0$ for 
$\beta\in\Phi^{+}$), 
$w,w'\in W^p$, $w\to w'$ in the parabolic Hasse graph 
for $(\lieg, \liep)$ and assume that $w\tilde\lambda,\, w'\tilde\lambda\in P_{\liep}^{++}+\delta$. 
Further, suppose that there exists a nonzero homomorphism of true Verma modules 
$M_{\lieb}(w'\tilde\lambda)\to M_{\lieb}(w\tilde\lambda)$. 
Then the standard homomorphism $M_{\liep}(w'\tilde\lambda)\to M_{\liep}(w\tilde\lambda)$ is nonzero.
\end{theorem}

\begin{remark}
In \cite{Lepowski}, the theorem is formulated only for $\tilde\lambda\in P^{++}+\delta$, 
but the proof works for non-integral $\tilde\lambda$ as well. 
Note, that for non-integral (and neither $\lieg$-, nor $\liep$-dominant) $\tilde\lambda-\delta$, the weights
$w\tilde\lambda-\delta$ and $w'\tilde\lambda-\delta$ may still be $\liep$-dominant and $\liep$-integral.
\end{remark} 

\begin{proof}
Assume that the standard homomorphism is zero. It follows from 
lemma \ref{zeromap} that there exists $\alpha\in \Delta-\Sigma$ so that 
$M_{\lieb}(w'\tilde\lambda)\subset M_{\lieb}(s_\alpha w \tilde\lambda)$.
The last statement of theorem $\ref{truevermamodulesmap}$ implies that
$w'> s_\alpha w$ in the Bruhat ordering. But, because $w\tilde\lambda\in P_{\liep}^{++}+\delta$
and $\alpha\in\Delta-\Sigma$, we have $(w\tilde\lambda)(H_\alpha)\in\N$ and
it follows from $\ref{truevermamodulesmap}$ that $M_{\lieb}(s_\alpha w\tilde\lambda)\subset
M_{\lieb}(w\tilde\lambda)$ and $l(s_\alpha w)=l(w)+1$. So we have $l(w')>l(s_\alpha w)>l(w)$
which contradicts $l(w')=l(w)+1$.
\end{proof}

For any weight $\lambda$, there always exists a dominant weight $\tilde\lambda$ 
(i.e. $\tilde\lambda(H_\beta)\geq 0$ for $\beta\in\Phi^+$)
on the same orbit of the Weyl group. 
If there exists some $\beta$ so that $\tilde\lambda(H_\beta)=0$, 
we say that the generalized Verma modules $M_{\liep}(w\tilde\lambda)$ have 
{\it singular character} and the weights $w\tilde\lambda$ are called {\it singular}. 
Theorem $\ref{genlepowski}$ cannot be generalized to singular weights, because 
for singular $\tilde\lambda$, the weight $w\tilde\lambda$ doesn't determine $w$ uniquely. 
(However, there are indications that a similar theorem may be
true, if we admit non-standard GVM homomorphisms.) 
%

The following lemma will be used for comparing lengths of two elements in $W^p$:

\begin{lemma}
\label{grading2length}
Let $E$ be the grading element for the pair $(\lieg, \liep)$ and let $w,w'\in W^p$, $w'=s_\gamma w$ and
$l(w')>l(w)$. Then $w\delta(E)-w'\delta(E)\in\N$. 
\end{lemma}

\begin{proof}
Because $w\in W^p$ and $w'=s_\gamma w\in W^p$, the uniqueness of the decomposition $W=W_p W^p$ yields $s_\gamma\notin W_p$. From the definition, $W_p=W_{\lieg_0}$, the Weyl group of $\lieg_0$, so the root $\gamma$ cannot be expressed as sum
of simple roots in $\Delta-\Sigma$. The definition of the grading $\oplus_j \lieg_j$ of $\lieg$, associated to
the pair $(\lieg, \liep)$ implies that the $\gamma$-root space generator $X_\gamma\in\lieg_i$ for some $i>0$, so $\gamma(E)=i\in\N$.
We obtain $w'\delta(E)=(s_\gamma w\delta)(E)=(w\delta-w\delta(H_\gamma) \gamma)(E)=w\delta(E)-i w\delta(H_\gamma)$.
Because $\delta$ is dominant and $l(w')>l(w)$, we have $w\delta(H_\gamma)>0$. The weight $\delta$ is also
integral, because $\delta(H_\alpha)=1$ for each $\alpha\in\Delta$. So the difference 
$(w\delta-w'\delta)(E)=i w\delta(H_\gamma)$ is a product of two positive integers.
\end{proof}

\subsection{Order of the differential operator dual to a GVM homomorphism}
The following theorem is an important tool for determining the order of an operator, dual to
a homomorphism of generalized Verma modules, if the highest weights of the inducing representations are known.

\begin{theorem}
\label{grading2degree}
Let $\mu, \lambda$ be highest weights of some irreducible finite-dimensional 
$P$-modules $\V_\mu, \, \V_\lambda$ and $\phi: M_{\liep}(\V_\mu)\to M_{\liep}(\V_\lambda)$ be a nonzero 
homomorphism of generalized Verma modules. Let $E$ be the grading element
for $(\lieg,\liep)$ and let $o:=(\lambda-\mu)(E)$. Then $o$ is an integer larger or equal to the 
order of the dual differential operator $\Gamma(G\times_P \V_\lambda^*)\to \Gamma(G\times_P \V_\mu^*)$.
Further, if $o\in\{1,2\}$, then $o$ is the order of the operator.
\end{theorem}

\begin{proof}
Let $v_\mu$ be the highest weight vector of $\V_\mu$ and 
$\phi(1\otimes v_\mu)=\sum_j y_j\otimes v_j$, $y_j\in\univ(\lieg_-)$, $v_j\in\V_\lambda$ ($M_{\liep}(\lambda)\simeq\univ(\lieg_-)\otimes\V_\lambda$ as vector space).
Let $k$ be the maximal integer so that $y_i\in\univ_k (\lieg_-)-\univ_{k-1}(\lieg_-)$ for some $y_i$ and let $0\neq g_0\in\univ(\lieg_0)$.
Then
$\phi$ maps $1\otimes g_0\cdot v_\mu=g_0\otimes_{\univ(\liep)}v_\mu$ to 
$$\sum_j g_0 y_j\otimes_{\univ(\liep)} v_j=\sum_j (y_j g_0 +[g_0,y_j])\otimes_{\univ(\liep)}v_j=
\sum_j (y_j\otimes g_0\cdot v_j + [g_0,y_j]\otimes v_j)$$
because for $g_0\in \univ(\lieg_0)$ and $y_j\in\univ(\lieg_-)$, $[g_0,y_j]\in\univ(\lieg_-)$ ($[a,b]=ab-ba$ is the
commutator in the associative algebra $\univ(\lieg)$). Simple commutation relations show that, if
$y_j\in\univ_{l}(\lieg_-)-\univ_{l-1}(\lieg_-)$, then  $[g_0, y_j]\in\univ_{l}(\lieg_-)-\univ_{l-1}(\lieg_-)$ as well. 
Therefore, $\phi$ maps 
$1\otimes g_0\cdot v_\mu$ into $\univ_{k}(\lieg_-)\otimes \V_\lambda$ but not to $\univ_{k-1}(\lieg_-)\otimes \V_\lambda$. 
$\V_\mu$ is an irreducible $\liep$-module and $\lieg_0$ is the reductive part of $\liep$, so 
so $\univ(\lieg_0)v_\mu=\V_\mu$ and $\phi$ maps $1\otimes \V_\mu$ into $\univ_{k}(\lieg_-)\otimes \V_\lambda$.
Let $v\in\V_\mu$, 
$\phi(1\otimes v)=\sum_j \tilde{y}_j\otimes \tilde{v}_j$,  $\tilde{v}_j\in\V_\lambda$, $\tilde{y}_j\in\univ_{k}(\lieg_-)$ and 
$\tilde{y}_i\notin \univ_{k-1}(\lieg_-)$ for some $i$. 
Let $\tilde{y}_j=y_1^{(j)}\ldots y_{l(j)}^{(j)}$
for some $y_u^{(j)}\in\lieg_-$, $l(j)\leq k$ and $l(i)=k$.

Applying the duality (\ref{duality}),
the differential operator $D$ satisfies
$$v((Df)(0))=\sum_j \tilde{v}_j(L_{y_1^{(j)}} \ldots L_{y_{l(j)}^{(j)}} (f)(0)),$$ where $L_{y_u^{(j)}}$ are the left
invariant vector fields generated by $y_u^{(j)}\in\lieg_-$. So, the operator $D$
dual to the homomorphism is of order $k$.

Let us suppose that the operator has order $k$, i.e. $\phi$ maps $1\otimes v_\mu$ into 
$\univ_k(\lieg_-)\otimes \V_\lambda$ but not into $\univ_{k-1}(\lieg_-)\otimes \V_\lambda$.
Let $\{y_1,\ldots y_n\}$ be an ordered basis of $\lieg_-$ that consists of generators of 
negative root spaces in $\lieg_-$.

Let $\phi(1\otimes v_\mu)=\sum_j \tilde{y}_j\otimes v_j$ and assume that all the $v_j$'s are 
weight vectors in $\V_\lambda$ and $\tilde{y}_j$ is a product of the $y_j$'s
(it follows from the PBW theorem that such expression is always possible). Then all $\tilde{y}_j\otimes v_j$ are weight vectors
and, because their sum is a weight vector of weight $\mu$, each $\tilde{y}_j\otimes v_j$ is
a weight vector of weight $\mu$ as well.

Because $\phi(1\otimes v_\mu)\notin \univ_{k-1}(\lieg_-)\otimes \V_\lambda$, there exists $i$ such that 
$\tilde{y}_i=y_{i_1}\ldots y_{i_k}$ is a product of $k$ elements. Let $u_j\in\N$ be defined by 
$y_{i_j}\in\lieg_{-u_j}$.
The action of the grading element on $y_{i_1}\ldots y_{i_k}\otimes v_i$ is
\begin{eqnarray*}
&&\mmmm E\cdot (y_{i_1}\ldots y_{i_k}\otimes v_i)=
Ey_{i_1}\ldots y_{i_k}\otimes_{\univ(\liep)}v_i=\\
&&\mmmm=(y_{i_1}E+[E,y_{i_1}])y_{i_2}\ldots y_{i_k}\otimes_{\univ(\liep)}v_i=\ldots=\\
&&\mmmm=y_{i_1}\ldots y_{i_k} (\lambda(E)-u_1-\ldots -u_k)\otimes v_i
\end{eqnarray*}
But $y_{i_1}\ldots y_{i_k}\otimes v_i$ is a weight vector of weight $\mu$, so the 
left hand side equals $\mu(E)(y_{i_1}\ldots y_{i_k}\otimes v_i)$. It follows
\begin{equation}
\label{lambda-mu}
(\lambda-\mu)(E)=\sum_j u_j\geq k
\end{equation}
because $u_j\geq 1$ for all $j$.
So, we see that $(\lambda-\mu)(E)$ is always an integer larger or 
equal to the order of the operator.

It follows immediately that $(\lambda-\mu)(E)=1$ implies that the operator is of first
order. To finish the proof, it remains to show that for a first order operator,
$(\lambda-\mu)(E)$ is $1$.

Assume that $D$ is an operator of first order. This means that 
$\phi(1\otimes v_\mu)=\sum_j y_j\otimes v_j$ for $y_j\in\univ_1(\lieg_-)$ and again,
assume that $y_j$ are either constants or generators of negative root spaces and $v_i$
are weight vectors. All the terms $y_j\otimes v_j$ are of weight $\mu$, and therefore, 
$$\mu(E)(y_j\otimes v_j)=E (y_j\otimes v_j)=(\lambda(E)+[E,y_j])(y_j\otimes v_j)$$
so $[E,y_j]=(\mu-\lambda)(E)$ for all $j$ and it follows that all the $y_j$'s are from the
same graded components of $\lieg$. If $y_j\in\lieg_{-1}$, so $(\lambda-\mu)(E)=1$ and we are done.
Assume, for contradiction, that $y_j\in\lieg_{-k}$ for $k>1$.

Because $\sum_j y_j\otimes v_j\in\lieg_{-k}\otimes \V_\lambda$, choosing a basis $\{\tilde{v}_1,\ldots, \tilde{v}_m\}$ of
$\V_\lambda$, $\sum_j y_j\otimes v_j$ can be uniquely expressed as 
$\sum_{j=1}^m \tilde{y}_j\otimes \tilde{v}_j$ for some $\tilde{y}_j\in\lieg_{-k}$.
Because it is 
a homomorphic image of a highest weight vector in $M_{\liep}(\mu)$, 
it must be annihilated by all positive root spaces in $\lieg$, in particular, by any generator $x$ of a root space
in $\lieg_1$:
\begin{eqnarray*}
&&\mmmm x\cdot (\sum_j \tilde{y}_j\otimes \tilde{v}_j)=\sum_j x \tilde{y}_j\otimes_{\univ(\liep)}\tilde{v}_j=\sum_j (\tilde{y}_jx+[x,\tilde{y}_j])
\otimes_{\univ(\liep)}\tilde{v}_j=\\
&&\mmmm =\sum_j (\tilde{y}_j\otimes_{\univ(\liep)}x\cdot \tilde{v}_j+[x,\tilde{y}_j]\otimes_{\univ(\liep)}\tilde{v}_j)=
\sum_j [x,\tilde{y}_j]\otimes \tilde{v}_j=0
\end{eqnarray*}
because $[x,\tilde{y}_j]\in \lieg_{-k+1}\subset\lieg_-$ and $x\cdot \tilde{v}_\lambda=0$.
Because $\tilde{v}_j$ forms a basis of $\V_\mu$, it follows that for each $j$, $[x,\tilde{y}_j]=0$
for all $x\in\lieg_1$. The grading fulfills that $\lieg_{-1}$
generates $\lieg_-$ and $\lieg_{1}$ generates 
$\liep^+=\sum_{i\geq 1} \lieg_{i}$. The Jacobi identity implies that
if $\tilde{y}_j$ commutes with $\lieg_1$, it commutes with all the $\liep^+$
as well. Let $\tilde{y}_j=\sum_i a_i y_{-\phi_i}$ where $y_{-\phi_i}$ is a generator
of the $-\phi_i$-root space. Define $x:=\sum_i a_i x_{\phi_i}$, where 
$x_{\phi_i}$ is a generator of the $\phi$-root space. We see that 
$x\in\lieg_{k}$ and $[x, \tilde{y}_j]=\sum_i a_i^2 [x_\phi, y_{-\phi}]\neq 0$ and we have a contradiction.
\end{proof}

\section{The orbits associated with the Dirac operator}
\subsection{Existence of the homomorphisms}

Let as suppose that $n$ is odd, $\lieg=B_{k+(n-1)/2}=\so(n+2k,\C)$, $\liep$ its parabolic subalgebra
corresponding to $$\dynkin\whiteroot{}\link\ldots\link\whiteroot{}\link\noroot{}\link\whiteroot{}\link\ldots\link\whiteroot{}\llink>\whiteroot{}\enddynkin$$
where the $k$-th node is crossed ($\Sigma=\{\alpha_k\}$).
The subalgebra $\liep$ induces the $2$-gradation
$
\label{2grading}
\lieg=
\left(
\begin{tabular}{c|c|c}
$\lieg_0$ & $\lieg_{1}$ & $\lieg_{2}$ \\
\hline
$\lieg_{-1}$ & $\lieg_0$ & $\lieg_{1}$ \\
\hline
$\lieg_{-2}$ & $\lieg_{-1}$ & $\lieg_0$
\end{tabular}
\right),
$
where $\lieg_0$ consists of blocks of dimension $k\times k$, $n\times n$ and $k\times k$.
The corresponding grading element is $E=\diag(1,\ldots,1,0,\ldots,0,-1\ldots,-1)$ 
and the action of a weight 
$[a_1,\ldots,a_k|b_1,\ldots, b_{(n-1)/2}]$ on E is $\sum_i a_i$.

In this section, we will try to describe the structure of GVM homomorphisms on the Weyl 
orbit of the weight
\begin{eqnarray*}
&&\mmmm \lambda=\dynkin\whiteroot{0}\link\ldots\link\whiteroot{0}\link\noroot{-\frac{n}{2}}\link
\whiteroot{0}\link\ldots\link\whiteroot{0}\llink>\whiteroot{1}\enddynkin +\delta.
\end{eqnarray*}
It was shown in \cite{srni, Tampere} that there exists a GVM homomorphism 
$M_{\liep}(\mu)\to M_{\liep}(\lambda)$ so that
the dual differential operator may be identified with the 
Dirac operator in various Clifford variables, as noticed in the introduction 
(choosing the real Lie groups $G=\Spin(n+k,k)$ and $P$ the parabolic subgroup so that its 
complexified Lie algebra is $\liep$).

Let as represent the elements of $\lieg$ as matrices antisymmetric with respect to the anti-diagonal, 
choose the Cartan algebra to be the algebra of diagonal matrices in 
$\lieg$ and a natural basis $\{\epsilon_i\}$
of $\lieh^*$ defined by $$\epsilon_i(\diag(a_1,\ldots, a_{k+(n-1)/2},0, -a_{k+(n-1)/2},\ldots, -a_1)):=a_i$$
(see e.g. \cite{GW} for details).

In the $\epsilon_i$-basis, $\delta=[\ldots,5/2,3/2,1/2],$ $\lieg$-dominant weights are those 
$[a_1,\ldots,a_{k+(n-1)/2}]$ such that $a_1\geq a_2\geq \ldots \geq a_{k+(n-1)/2}\geq 0$
and $\liep$-dominant weights must fulfill $a_1 \geq a_2\geq\ldots\geq a_k$
and $a_{k+1}\geq \ldots\geq a_{k+(n-1)/2}\geq 0$. 
A weight is $\liep$-dominant and $\liep$-integral, if, moreover, $a_i-a_j\in\Z$ for $i,j\leq k$ and 
$a_l\in \Z/2$ for $l>k$. Positive roots are all $[0,\ldots,0, 1,0,\ldots, -1,\ldots]$, $[\ldots,1,\ldots,1,\ldots]$ and $[\ldots, 0,1,0,\ldots]$.
The corresponding root reflections map the weight $[\ldots, a_i, \ldots, a_j,\ldots]$ to $[\ldots, a_j, \ldots, a_i, \ldots]$ (transpositions), 
or to $[\ldots, -a_j,\ldots, -a_i, \ldots ]$ (sign-transpositions) or to $[\ldots, -a_i,\ldots, a_j,\ldots ]$ (sign-change).

The weight $\lambda$ we consider can be writen in the $\epsilon_i$-basis as 
$$\lambda=[(2k-1)/2,\ldots,3/2 ,1/2|\ldots,3,2,1].$$

\begin{lemma}
\label{2diraksodd}
Let $k=2$. Then there exist three nonzero weights $\mu, \nu, \xi\in P_{\liep}^{++}$ on the orbit of $\lambda$ and 
nonzero standard homomorphisms 
$$M_{\liep}(\xi)\to M_{\liep}(\nu)\to M_{\liep}(\mu)\to M_{\liep}(\lambda),$$ where the weights are 
described by the following diagram: 

\centerline{
\includegraphics{bgg.14}
}
\end{lemma}

\begin{proof}
The existence of true Verma module homomorphisms $M_{\lieb}(\xi)\to\ldots\to M_{\lieb}(\lambda)$ 
follows easily from Theorem $\ref{truevermamodulesmap}$. All the weights are from $P_{\liep}^{++}+\delta$
and they are on the orbit of the $\lieg$-dominant weight $\tilde\lambda=[\ldots,4,3,2,3/2,1,1/2]$. This weight is 
nonsingular, because its coefficients are strictly decreasing and the last one is strictly positive. 

Let $w$ resp. $w', w'', w'''$ be the elements of $W$ that takes $\tilde\lambda$ to $\lambda$ resp. $\mu, \nu, \xi$.
Easy calculation shows that $w$ can be characterized by $w\delta=[5/2, 1/2| \ldots, $ $9/2,7/2,3/2]$ 
and $w'\delta=[5/2, -1/2|\ldots,9/2,7/2,3/2]$. Because $w'$ and $w$ are connected by a root reflection,
lemma $\ref{path}$ states that either $w\leq w'$ ot $w'\leq w$ in the Bruhat ordering and there exists a sequence
$w=w_0\to w_1\to \ldots \to w_{n-1}\to w_n=w'$, $w_i\in W^p$. Lemma $\ref{grading2length}$ states $(w_i\delta-w_{i+1}\delta)(E)\in\N$ for all $i$, where $E$ is the grading element. But we compute 
$(w\delta-w'\delta)(E)=(5/2+1/2)-(5/2-1/2)=1$, so the only possibility is $n=1$ and $w\to w'$. Applying $\ref{genlepowski}$, we see that the standard map \gvmhom is nonzero.

The element $w''$ takes $\delta$ to $[1/2,-5/2|\ldots, 9/2,7/2,3/2]$ and $(w''\delta-w'\delta)(E)=(5/2-1/2)-(1/2-5/2)=4$. The lenth difference $l(w'')-l(w')$ must be odd, because $w''=s_\gamma w'$ for $\gamma=[1,1|0,\ldots,0]$, and a root reflection has negative determinant. So either $w'\to w''$, or $w'\to w_1\to w_2\to w''$. In the first case, we apply theorem $\ref{genlepowski}$ as before. Suppose $w'\to w_1\to w_2\to w''$ and suppose, for contradiction, that the standard homomorpism $M_{\liep}(\nu)\to M_{\liep}(\mu)$ is zero. Theorem \ref{zeromap} says that the true Verma modules
\begin{equation}
\label{condition}
M_{\lieb}(\nu)\subset M_{\lieb}(s_\alpha\mu)
\end{equation}
for some simple root $\alpha\neq \alpha_2$. 
We know that for such $\alpha$, $s_\alpha\in W_p$ and, because $\mu$ is $\liep$-dominant, $M_{\lieb}(s_\alpha\mu)\subsetneq M_{\lieb}	(\mu)$.
The weight $s_\alpha\mu$ is one of the following types:
\begin{enumerate}
\item{$[-1/2,3/2|\ldots,3,2,1]$ if $\alpha=\alpha_1$}
\item{$[3/2,-1/2|(n-1)/2,\ldots,l-1,l,\ldots,2,1]$ if $\alpha=\alpha_i$, $2<i<k+(n-1)/2$} 
\item{$[3/2,-1/2|\ldots,3,2,-1]$ if $\alpha=\alpha_{k+(n-1)/2}$} 
\end{enumerate}
First we show that $\alpha\neq\alpha_1$. If $\alpha=\alpha_1$, (\ref{condition}) implies 
that $s_{\alpha_1}\mu-\nu=[-1,3|0,\ldots,0]$ is a sum of positive roots, but this is not possible, as no positive root
is of the form $[-1,\hbox{something}]$.

Now assume that $s_\alpha \mu$ is of type $(2)$. Because
$$
M_{\lieb}(w''\tilde\lambda)=M_{\lieb}(\nu)\subsetneq M_{\lieb}(s_\alpha\mu)\subsetneq M_{\lieb}(\mu)=M_{\lieb}(w'\tilde\lambda),
$$
$l(w')-l(w)=3$ and $\nu$ is not connected with $s_\alpha\mu$ by any root reflection, it follows from
Theorem \ref{truevermamodulesmap} that there must be $\beta_1, \beta_2$ so that
\begin{equation}
\label{vermasubsets}
M_{\lieb}(\nu)\subsetneq M_{\lieb}(s_{\beta_1}\nu)=M_{\lieb}(s_{\beta_2} s_\alpha\mu)\subsetneq M_{\lieb}(s_\alpha\mu).
\end{equation}

Note, that the weights are $s_\alpha \mu=[3/2,-1/2|\ldots,l-1,l,\ldots,2,1]$ and 
$\nu=s_{\beta_1} s_{\beta_2} s_\alpha\mu=[1/2,-3/2|\ldots,2,1]$. In coordinates, $s_{\beta_j}$ cannot be a (sign)-transposition
interchanging an integer and a half-integer, because of the conditions $s_\alpha \mu(H_{\beta_2})\in\N$ and 
$s_{\beta_2}s_\alpha\mu(H_{\beta_1})\in\N$.
So, exactly one of these reflections
interchanges $(3/2,-1/2)$ to $(1/2,-3/2)$ and the other one interchanges $(l-1,l)$ to 
$(l,l-1)$. So either $s_{\beta_2}s_\alpha \mu=[1/2,-3/2|\ldots,l-1,l\ldots]$ or
$s_{\beta_2}s_\alpha \mu=[3/2,-1/2|\ldots,l,l-1,\ldots]$. In the first case, 
$M_{\lieb}(s_{\beta_2} s_\alpha\mu)=M_{\lieb}(s_{\alpha}\nu)\subsetneq M_{\lieb}(\nu)$ ($\nu$ is $\liep$-dominant) which contradicts $(\ref{vermasubsets})$. 
In the second case, $M_{\lieb}(s_{\beta_2}s_\alpha \mu)=M_{\lieb}(\mu)\subsetneq M_{\lieb}(s_\alpha\mu)$ 
by $(\ref{vermasubsets})$, which contradicts the fact that $M_{\lieb}(s_\alpha\mu)\subsetneq M_{\lieb}(\mu)$. 
So $s_{\alpha}\mu$ cannot be of type $(2)$.

Similarly, we can show that $s_\alpha (\mu)$ cannot be of type $(3)$. But this means that $(\ref{condition})$
does not hold and the standard map $M_{\liep}(\nu)\to M_{\liep}(\mu)$ is nonzero.

Finally, note that $w'''\delta=[-1/2,-5/2|\ldots]$, so $(w''\delta-w''\delta)(E)=(1/2-5/2)-(-1/2-5/2)=1$, therefore $w''\to w'''$ and the standard homomorphism $M_{\liep}(\xi)\to M_{\liep}(\nu)$ is nonzero.
\end{proof}

If $n\neq 5$, there are no other weights from $P_{\liep}^{++}+\delta$ on the orbit of $\tilde\lambda$. 
In case $n=5$,
there are other weights $[2,1|3/2,1/2]$, $[2,-1|3/2,1/2]$, $[1,-2|3/2,1/2]$  and $[-1,-2|3/2,1/2]$ on this orbit, 
but there is no nonzero homomorphism from the GVM's in the last theorem to any of these and vice versa.

\begin{theorem}
\label{2diroddcomplex}
The sequence of homomorphisms $M_{\liep}(\xi)\to M_{\liep}(\nu)\to M_{\liep}(\mu)\to M_{\liep}(\lambda)$ is a complex.
\end{theorem}

\begin{proof}
We want to show that the standard homomorphism $M_{\liep}(\nu)\to M_{\liep}(\lambda)$ is zero. This can be using theorem $\ref{zeromap}$ and the facts that 
$$M_{\lieb}([\half,-\frac{3}{2}|\ldots, 2,1])\subset M_{\lieb} ([\half,\frac{3}{2}|\ldots,2,1])=M_{\lieb}(s_{\alpha_1} [\frac{3}{2}, \half|\ldots,2,1]).$$
Similarly, we could show that $M_{\liep}(\xi)\to M_{\liep}(\mu)$ is zero.
\end{proof}

\begin{definition}
\label{S_k}
Let as define an oriented graph $S_k$ in the following way: $S_1$ has $2$ vertices connected by an arrow 
($\bullet\to\bullet$), $S_2$ contains $4$ vertices connected linearly by arrows $(\bullet\to\bullet\to\bullet\to\bullet)$. 
For $k\geq 3$, $S_k$ contains $2$ disjoint subsets 
$S^1$ and $S^2$ of vertices so that the subgraphs $S^1$and $S^2$ are both isomorphic to $S_{k-1}$, where $S^1$ contains the ``first'' vertex and $S^2$ the ``last'' one.  
Similarly, $S^1$ contains $2$ copies of $S_{k-2}$, denote them by $S^{1,1}$ and $S^{1,2}$ and $S^2$ contains $2$ copies of $S_{k-2}$, denote them by $S^{2,1}$ and $S^{2,2}$. Let $\phi$ resp. $\psi$ be the
isomorphism $S_{k-2}\to S^{1,2}$  resp. $S_{k-2}\to S^{2,1}$. Then for each vertex $x\in S_{k-2}$ there is an arrow $\phi(x)\to\psi(x)$ in $S_k$. For completeness, define $S_0$ to be a one-point graph.

Graphically,  $S_k$ has the following structure:

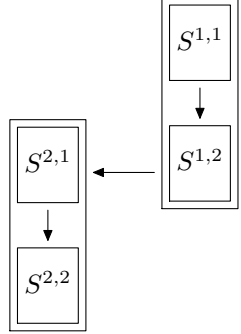
\centerline{
\includegraphics{singorb.6}
}

\end{definition}

We draw the graphs $S_k$ for $k=3, 4$:

\centerline{
\includegraphics{bgg.17}
\hspace{2cm}
\includegraphics{bgg.18}
}

\begin{theorem}
\label{kdirgvmoddorb}
Let $(\lieg,\liep)$ and $\lambda$ be like at the beginning of this section and let $k\neq (n-1)/2$.
There are $2^k$ weights from $(P_{\liep}^{++}+\delta)\cap W\lambda$ and they can be assigned to the vertices of the graph $S_k$ so that for each arrow $\mu\to\nu$ in this graph there exists a nonzero standard homomorphism 
$M_{\liep}(\nu)\to M_{\liep}(\mu)$ and each nonzero standard homomorphism between GVM's with highest weights from 
$((P_{\liep}^{++}+\delta)\cap W\lambda)-\delta$
is a composition of these.
The weight $\lambda$ itself is assigned to the minimal vertex in $S_k$.
\end{theorem}
\begin{proof}
The condition on a weight $\nu=[a_1,\ldots,a_k|b_1,\ldots,b_{(n-1)/2}]$ to be from $P_{\liep}^{++}+\delta$ 
is $a_1>\ldots>a_k$, $b_1>\ldots>b_{(n-1)/2}>0$, $a_i-a_j\in\Z$, $b_i-b_j\in\Z$ and the $b_i$'s are all
integers or all half-integers. Simple combinatorics implies that, if $\nu\in P_{\liep}^{++}+\delta$ 
is on the orbit of $\lambda$ and $k\neq (n-1)/2$, the only possibility is 
$\nu=[a_1,\ldots,a_k|(n-1)/2,\ldots,2,1]$, where $(a_1,\ldots,a_k)$
is some strictly decreasing sign-permutation of $((2k-1)/2,\ldots,3/2,1/2)$.

These conditions imply that there is either $(2k-1)/2$ on the first position, or $-(2k-1)/2$ on the $k$-th
position and the remaining of the first $k$ positions contains a decreasing sign-permutation of $((2k-3)/2,\ldots,1/2)$. This proves that there are $2^k$ such weights. Define $R_k$ to be the set of these weights, $R_k^1$ to be the set of weights with $(2k-1)/2$ on the first position and $R_k^2$ to be the set of weights with $-(2k-1)/2$ on the $k$-th position.

We will prove that the map $i: R_{k-1}\to R_k^1$ given by 
$([a_1,\ldots,a_{k-1}|\ldots])\mapsto ([(2k-1)/2,a_1,\ldots,a_{k-1}|\ldots])$
preserves the existence of nonzero standard GVM homomorphisms (i.e. there exists a nonzero standard $M_{\liep_{k-1,n}}(\nu)\to M_{\liep_{k-1,n}}(\mu)$ if and only if there exists a nonzero standard 
$M_{\liep_{k,n}}(i(\nu))\to M_{\liep_{k,n}}(i(\mu))$, the subscripts $k,n$ means that the rank of the 
Lie algebra is $k+(n-1)/2$). 

We start with the Borel case $\liep=\lieb$. 
Let $M_{\lieb_{k-1,n}}(\nu)\to M_{\lieb_{k-1,n}}(\mu)$ be a true Verma module homomorphism. 
Let as denote by $\tilde{i}$ the map $\lieh_{k-1,n}^*\to\lieh_{k,n}^*$ defined by
$[a_1,\ldots,a_{k-1}|b_1,\ldots,b_{(n-1)/2}]\mapsto [0,a_1,\ldots,a_{k-1}|b_1,\ldots,b_{(n-1)/2}]$. 
According to $\ref{truevermamodulesmap}$, there exists a nonzero homomorphism
$M_{\lieb_{k-1,n}}(\nu)\to M_{\lieb_{k-1,n}}(\mu)$ if and only if there exists a sequence 
$\mu=\mu_0, \mu_1, \ldots, \mu_l=\nu$ of weights connected by root reflections so that 
$\mu_j-\mu_{j-1}$ is a positive integral multiple of a positive root from $\Phi_{k-1,n}^+$ 
(this is the set of positive roots of $\lieg=\so(2(k-1)+n)$) for all $j$. 
In this case, the sequence $i(\mu)={i}(\mu_0), {i}(\mu_1), \ldots, {i}(\mu_l)={i}(\nu)$ 
has similar properties, because $\mu_j=s_\gamma \mu_{j-1}$ implies 
${i}(\mu_j)=s_{\tilde{i}(\gamma)} i(\mu_{j-1})$ and for each 
$\gamma\in\Phi_{k-1,n}^+$, $\tilde{i}(\gamma)\in\Phi_{k,n}^+$. 
So, there exists a nonzero homomorphism 
$M_{\lieb_{k,n}}(i(\nu))\to M_{\lieb_{k,n}}(i(\mu))$. 
On the other hand, if there exists a nonzero  homomorphism 
$M_{\lieb_{k,n}}(i(\nu))\to M_{\lieb_{k,n}}(i(\mu))$, it follows that there is a sequence 
$i(\mu)=[(2k-1)/2, \hbox{something}]=i(\mu_0), i(\mu_1)$, $\ldots$, $i(\mu_l)=[(2k-1)/2, \hbox{something}]$, $i(\mu_j)=s_{\gamma_j}i(\mu_{j-1})$, so that $i(\mu_j)-i(\mu_{j-1})$ 
is a positive multiple of a positive root. Therefore, the coefficient on the first position 
is not increasing in this sequence: so, it is constant $(2k-1)/2$. 
This means that the root reflections $\gamma_j$ don't interchange the first coordinate with some other and the roots $\gamma_j$ have zeros on first positions. So, there exist $\tilde{\gamma_j}\in\Phi_{k-1,n}^+$ so that 
$\tilde{i}\tilde{\gamma_j}=\gamma_j$ and we obtain that there exists a nonzero homomorphism 
$M_{\lieb_{k-1,n}}(\nu)\to M_{\lieb_{k-1,n}}(\mu)$.

It follows from Theorem \ref{zeromap} that the standard homomorphism  
$M_{\liep_{k-1,n}}(\nu)\to M_{\liep_{k-1,n}}(\mu)$ is zero
if and only if $M_{\lieb_{k-1,n}}(\nu)\subset M_{\lieb_{k-1,n}}(s_{\alpha_j}\mu)$ 
for some simple root
$\alpha_j\neq \alpha_{k-1}$. Then 
$M_{\lieb_{k,n}}(i(\nu))\subset M_{\lieb_{k,n}}(s_{\tilde{i}(\alpha_{j})} i(\mu))$
follows from the previous paragraph, $\tilde{i}(\alpha_j)\neq \alpha_k$ 
and the standard homomorphism $M_{\liep}(i(\nu))\to M_{\liep}(i(\mu))$ 
is zero as well. 
On the other hand, if $M_{\liep}(i(\nu))\to M_{\liep}(i(\mu))$ is zero, 
then $M_{\lieb_{k,n}}(i(\nu))\subset M_{\lieb_{k,n}}(s_{\alpha_i} i(\mu))$
for some simple root $\alpha_i\neq \alpha_k$. 
If $i=1$, then $M_{\lieb_{k,n}}(i(\nu))\subset M_{\lieb_{k,n}}(s_{\alpha_1}i(\mu))$ 
implies $s_{\alpha_1}(i(\mu))-i(\nu)$ is a sum of positive roots. 
But $i(\nu)$ contains $(2k-1)/2$ on the first position and $s_{\alpha_1}(i(\mu))$ contains
a number strictly smaller then $(2k-1)/2$ on the first position, what is a contradiction.
Therefore, $i>1$ and there is a $\alpha\in\Delta_{k-1,n}$, $\alpha\neq \alpha_{k-1}$ so that 
$\tilde{i}(\alpha)=\alpha_i$. Then $M_{\lieb_{k-1,n}}(\nu)\subset M_{\lieb_{k-1,n}}(s_{\alpha}\mu)$,
and the map $M_{\liep}(\nu)\to M_{\liep}(\mu)$ is zero as well.

We see that for any $\mu,\nu\in R_{k-1}$, there exists a nonzero standard GVM homomorphism 
$M_{\liep_{k-1,n}}(\nu)\to M_{\liep_{k-1,n}}(\mu)$ if and only if there exists a nonzero
standard homomorphism $M_{\liep_{k,n}}(i(\nu))\to M_{\liep_{k,n}}(i(\mu))$. 
Similarly, we can define the map $j: R_{k-1}\to R_k^2$ by $[a_1,\ldots,a_{k-1}|\ldots]\mapsto [a_1,\ldots,a_{k-1},-(2k-1)/2|\ldots]$
and prove that there exists a nonzero standard GVM homomorphism 
$M_{\liep_{k-1,n}}(\nu)\to M_{\liep_{k-1,n}}(\mu)$ if and only if there exists a nonzero
standard homomorphism $M_{\liep_{k,n}}(j(\nu))\to M_{\liep_{k,n}}(j(\mu))$. 


Let as now denote the maps $i$ and $j$ described before by $i_k$ and $j_k$, specifying the dimension of the (resulting) 
weights.
It remains to prove that for each $x\in R_{k-2}$ there exists a nonzero standard GVM homomorphism 
$M_{\liep_{k,n}}(j_{k}i_{k-1}(x))\to M_{\liep_{k,n}}(i_k j_{k-1}(x))$. 
In other words, we want to show that for any decreasing sign-permutation $(a_2,\ldots,a_{k-1})$ of 
$((2k-5)/2,\ldots,1/2)$, there exists a nonzero standard homomorphism $M_{\liep}(\nu)\to M_{\liep}(\mu)$, where
$$\nu=[\frac{2k-3}{2},a_2,\ldots,a_{k-1},-\frac{2k-1}{2}|\ldots,2,1],$$
$$\mu=[\frac{2k-1}{2},a_2,\ldots, a_{k-1}, -\frac{2k-3}{2}|\ldots,2,1].$$
It follows from $\ref{truevermamodulesmap}$ that there is a homomorphism of the corresponding true Verma modules 
(the weights are connected by the root reflection with respect to $[1,0,\ldots,0,1|0,\ldots,0]$.)

There is a unique $\lieg$-dominant weight $\tilde\lambda$ on the orbit of $\lambda$:
$\tilde\lambda =[(n-1)/2, (n-1)/2-1, \ldots, k, k-1/2, k-1, k-3/2,\ldots, 3/2,1,1/2]$
in case $(n-1)/2\geq k$ and $\tilde\lambda=
[k-1/2, k-3/2,\ldots, n/2, n/2-1/2, n/2-1,\ldots,1,1/2]$ in case $(n-1)/2<k$.

Let $w$ resp. $w'$ be the Weyl group element taking $\tilde\lambda$ to
$\mu$ resp. $\nu$. Simple computations show that, if $(n-1)/2\geq k-1$, then 
$w$ takes $\delta=\half[\ldots,5,3,1]$ to 
$\half[4k-3,b_2,\ldots,b_{k-1},-(4k-7)|\ldots]$ where $(b_2,\ldots,b_{k-1})$ 
is some decreasing sign-permutation of $((4k-11)/2,\ldots,5/2,1/2)$ and $w'$ 
takes $\delta$ to $\half[4k-7,b_2,\ldots,b_{k-1},-(4k-3)|\ldots]$.
The difference of the grading element evaluation
is then $(w\delta-w'\delta)(E)=\half((4k-3)-(4k-7)+\sum_j b_j)-\half((4k-7)-(4k-3)+\sum_j b_j)=4$.
If $(n-1)/2<k-1$, then $w$ takes $\delta(E)$ to $[k+n/2-1,\ldots,-(k+n/2-2)|\ldots]$, $w'$ takes
$\delta$ to $[k+n/2-2,\ldots,-(k+n/2-1)|\ldots]$ and $(w\delta-w'\delta)(E)=2$ in this case.

So, in either case, $(w\delta-w'\delta)(E)\leq 4$
and, similarly as in the proof of lemma \ref{2diraksodd}, either $w\to w'$ or
$w\to w_1\to w_2\to w'$. If $w\to w'$, we apply Theorem \ref{genlepowski} and 
see that there is a nonzero standard homomorphism $M_{\liep}(\nu)\to M_{\liep}(\mu)$.

Let $w\to w_1\to w_2\to w'$ and
assume, for the sake of contradiction, that the standard map
$M_{\liep}(w'\tilde\lambda)\to M_{\liep}(w\tilde\lambda)$ is zero. Therefore,
\begin{equation}
\label{cond2}
M_{\lieb}(\nu)=M_{\lieb}(w'\tilde\lambda)\subset M_{\lieb}(s_\alpha w\tilde\lambda)=M_{\lieb}(s_\alpha\mu)
\end{equation}
for some simple root $\alpha\neq\alpha_k$.

The weight $s_\alpha(\mu)$ is one of the following types:
\begin{enumerate}
\item$[a_2,(2k-1)/2,\ldots, a_{k-1}, -(2k-3)/2|\ldots,3,2,1]$ if $\alpha=\alpha_1$
\item$[(2k-1)/2,\ldots,a_{l},a_{l-1},\ldots,-(2k-3)/2|\ldots]$ if $\alpha=\alpha_j$, $1<j<k-1$
\item$[(2k-1)/2,\ldots,-(2k-3)/2,a_{k-1}|\ldots]$ if $\alpha=\alpha_{k-1}$
\item$[(2k-1)/2,\ldots,-(2k-3)/2|(n-1)/2,\ldots,l-1,l,\ldots,2,1]$ if $\alpha=\alpha_j$, $k<j<k+(n-1)/2$
\item$[(2k-1)/2,\ldots,-(2k-3)/2|\ldots,3,2,-1]$ if $\alpha=\alpha_{k+(n-1)/2}$ 
\end{enumerate}

First we show that it is not of type $(1)$. If $\alpha=\alpha_1$, (\ref{cond2}) implies 
$s_{\alpha_1}(\mu)-\nu$ is a sum of positive roots, i.e. 
$$[a_2, (2k-1)/2,\ldots,-(2k-3)/2|\ldots]-[(2k-3)/2,a_2,\ldots,-(2k-1)/2|\ldots]\in\N\Phi^+,$$ 
where $a_2\leq(2k-5)/2$.
But the difference cannot be obtained as a sum of positive roots, 
because it contains a negative number $a_2-(2k-3)/2$ on the first position.

Now assume that $s_\alpha(\mu)$ is of type $(2)-(5)$. 
Because
$$
M_{\lieb}(w'\tilde\lambda)=M_{\lieb}(\nu)\subsetneq M_{\lieb}(s_\alpha\mu)\subsetneq 
M_{\lieb}(\mu)=M_{\lieb}(w\tilde\lambda),
$$
$l(w')-l(w)=3$ and $\nu$ is not connected to $s_\alpha(\mu)$ by any root reflection, 
it follows from
Theorem \ref{truevermamodulesmap} that there must be $\beta_1, \beta_2$ so that
\begin{equation}
\label{subsetsgenodd}
M_{\lieb}(\nu)\subsetneq M_{\lieb}(s_{\beta_1}\nu)=M_{\lieb}(s_{\beta_2} s_\alpha\mu)\subsetneq 
M_{\lieb}(s_\alpha\mu)
\end{equation}
Similarly as in the proof of lemma \ref{2diraksodd}, we will show that $\alpha$ cannot be
of type $(2)-(5)$. Let $\alpha$ be of type $(2)$, i.e. 
\begin{eqnarray*}
s_\alpha(\mu)&=&[(2k-1)/2,\ldots,a_{l},a_{l-1},\ldots,-(2k-3)/2|\ldots],\\
\nu&=&[(2k-3)/2,\ldots,a_{l-1},a_l,\ldots,-(2k-1)/2)|\ldots].
\end{eqnarray*}
The root reflections $s_{\beta_1}$ and $s_{\beta_2}$ cannot interchange an integer with a half-integer, because
of the integrality conditions $s_\alpha(\mu)(H_{\beta_2})\in\N$ and 
$s_{\beta_2}s_\alpha(\mu)(H_{\beta_1})\in\N$.
There are two possibilities: either $s_{\beta_2}$ 
interchanges $a_l$ with $a_{l-1}$ and $s_{\beta_1}$ sign-interchanges
$((2k-1)/2,-(2k-3)/2)$ with $((2k-3)/2,-(2k-1)/2)$ on the particular positions, or $s_{\beta_2}$ sign-interchanges
$((2k-1)/2,-(2k-3)/2)$ with $((2k-3)/2,-(2k-1)/2)$ and
$s_{\beta_1}$ interchanges $a_l$ with $a_{l-1}$.
In the first case, $\beta_2=\alpha$ and (\ref{subsetsgenodd}) implies $M_{\lieb}(\mu)\subsetneq M_{\lieb}(s_\alpha\mu)$, 
which contradicts the fact that $M(s_\alpha\mu)\subsetneq M(\mu)$
for a simple root $\alpha\neq\alpha_k$ and $\mu\in P_{\liep}^{++}+\delta$.
In the second case, $\beta_1=\alpha$ and (\ref{subsetsgenodd}) implies
$M_{\lieb}(\nu)\subsetneq M_{\lieb}(s_\alpha\nu)$, which also contradicts $M_{\lieb}(s_\alpha\nu)\subsetneq M_{\lieb}(\nu)$.

Let $\alpha$ be of type $(3)$, i.e.
\begin{eqnarray*}
s_\alpha(\mu)&=&[(2k-1)/2,\ldots,-(2k-3)/2,a_{k-1}|\ldots],\\
\nu&=&[(2k-3)/2,\ldots,a_{k-1},-(2k-1)/2)|\ldots]
\end{eqnarray*}
If either $\beta_1=\alpha$ or $\beta_2=\alpha$, we get contradiction similarly as in case $(2)$. 
But there is no other possibility, because the $a_{k-1}$ on the $k$-th position has to move somehow
to the $(k-1)$-th position: if $\beta_2$ would fix it, then $\beta_1=\alpha$, if $\beta_2$ would take it
to the $(k-1)$-th position, then $\beta_2=\alpha$ and if $\beta_2$ would take it (possibly with a minus sign) 
to the $l$-th position for $l\neq k,k-1,1$, then $\beta_1$ has to (sign-) interchange the $l$-th and $(k-1)$-th
position, so $s_{\beta_1} s_{\beta_2}$ would fix the $(2k-1)/2$ on the first position, which is impossible.
The last possibility is $l=1$: this would mean that $\beta_2$ takes $a_{k-1}$ to the first position (possibly with a minus sign),
but $|a_{k-1}|<(2k-3)/2$ implies that $s_{\beta_2}s_\alpha(\mu)$ has a smaller number on the first position
as $\nu$ and $s_{\beta_2}s_\alpha(\mu)-\nu$ is not expressible as a sum of positive roots. 
This contradicts $M_{\lieb}(\nu)\subsetneq M_{\lieb}(s_{\beta_2}s_\alpha\mu)$.

In case $(4)$, we have 
\begin{eqnarray*}
s_\alpha(\mu)&=&[(2k-1)/2,\ldots,-(2k-3)/2|(n-1)/2,\ldots,l-1,l,\ldots,2,1]\\
\nu&=&[(2k-3)/2,\ldots,-(2k-1)/2)|(n-1)/2,\ldots,l,l-1,\ldots,2,1]
\end{eqnarray*}
Because the reflections with respect to $\beta_1,\beta_2$ cannot interchange an integer and a 
half-integer, it follows that one of them interchanges $l$
with $l-1$, so either $\beta_1=\alpha$ or $\beta_2=\alpha$ 
and we get a contradiction as in case $(2)$. The same happens in case $(5)$.

In either case, we get a contradiction, so
the standard map $M_{\liep}(\nu)\to M_{\liep}(\mu)$
is nonzero. 

So, we can assign  weights from $R_k$ to the vertices of the graph $S_k$ so that 
we assign the weights from $R^1$ to $S^1$, the weights from $R^2$ to $S^2$ 
and the proof follows by induction.

Finally, it is easy to check that any possible nonzero standard GVM homomorphisms on the orbit 
is a composition of the homomorphisms described above by reducing this problem to true Verma module homomorphisms and
considering theorem \ref{truevermamodulesmap}. 
\end{proof}

In case $k=(n-1)/2$, all the GVM homomorphisms described in the last theorem exist as well, but the whole orbit contains also
weights of type $[\ldots,2,1|(2k-1)/2,\ldots,3/2,1/2]$. There is no nonzero GVM homomorphism $M_{\liep}(\nu)\to M_{\liep}(\mu)$ 
where $\mu$ is of such type and $\nu$ of the type $[\ldots,3/2,1/2|\ldots,2,1]$ (or opposite). 

\subsection{Orders of the operators}
\begin{theorem}
All the operators dual to the homomorphisms described in theorem \ref{kdirgvmoddorb} have order $1$ or $2$. For any $k$,
the connecting operators $\phi(x)\to \psi(x)$ (described in definition \ref{S_k}) have order $2$ and the graph homomorphisms
$S_{k-1}\to S_k^1$ and $S_{k-1}\to S_k^2$ respect orders. This determines, by induction, all the order of all the operators. 
\end{theorem}

If we draw a line for first order operators and a double-line for second order operators in the diagrams, we obtain the following pictures:

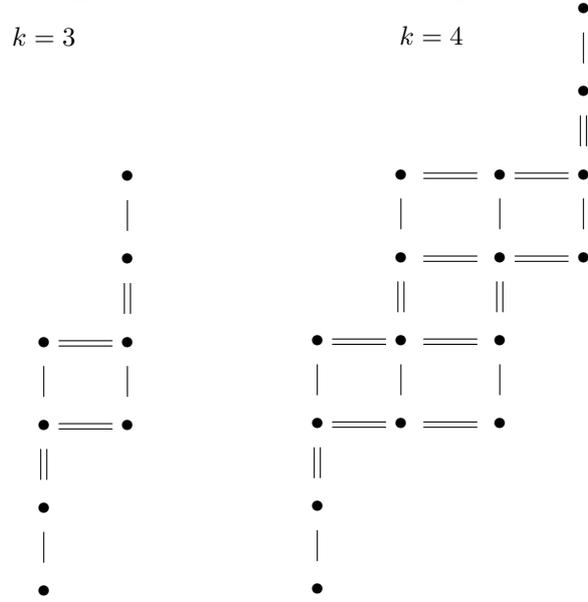
\centerline{
\includegraphics{bgg.7}
\hspace{2cm}
\includegraphics{bgg.8}
}

\begin{proof}
Recall that the action of a weight on the grading element is 
$$[a_1,\ldots, a_k|b_1,\ldots, b_{(n-1)/2}](E)=\sum_j a_j.$$
Applying theorem \ref{grading2degree} and the knowledge of the highest weights of the particular representations, we see
that 
\begin{eqnarray*}
&&\mmmm
[(\frac{2k-1}{2}, a_2,\ldots, a_{k-1}, -\frac{2k-3}{2}|\ldots](E)-
[\frac{2k-3}{2},a_2,\ldots,a_{k-1}, -\frac{2k-1}{2}|\ldots](E)=\\
&&\mmmm
=(\frac{2k-1}{2}-\frac{2k-3}{2})-(\frac{2k-3}{2}-\frac{2k-1}{2})=2,
\end{eqnarray*}
so the ``connecting'' operators are of second order. The other operators are of first order, because
\begin{eqnarray*}
&&\mmmm
[a_1,\ldots,a_{j-1},\half,a_{j+1},\ldots|\ldots](E)-[a_1,\ldots,a_{j-1},-\half,a_{j+1}\ldots|\ldots](E)=\\
&&\mmmm
\half-(-\half)=1.
\end{eqnarray*}

\end{proof}

\end{document}

%% file: makra.tex
\usepackage[cp1250]{inputenc}
\usepackage{amstext,amsbsy,amsmath,amscd,amssymb, graphics}
\usepackage[all]{xy}
\input{xy}
\input dynkin2bila
\advance\topmargin by -0.6cm

\numberwithin{equation}{section}

\newtheorem{Theo}{Theorem}[subsection]
\newtheorem{Cor}[Theo]{Corollary}
\newtheorem{Lem}[Theo]{Lemma}
\newtheorem{Exa}[Theo]{Example}

\newtheorem{Def}[Theo]{Definition}
\newtheorem{Rem}[Theo]{Remark}
\newtheorem{corollary}[Theo]{Corollary}

\let\cal=\mathcal

\def\Hom{\mathop {\rm Hom} \nolimits}

\def\R{\mathbb R}
\def\C{\mathbb C}
\def\N{\mathbb N}

\def\Z{\mathbb Z}
\def\V{\mathbb V}
\def\W{\mathbb W}
\def\S{\mathbb S}

\def\lieg{\mathfrak{g}}
\def\lieh{\mathfrak{h}}
\def\liep{\mathfrak{p}}
\def\lieb{\mathfrak{b}}

\def\univ{{\cal U}}
\def\so{{\mathfrak{so}}}

\def\Spin{{\mathrm{Spin}}}

\def\gl{{\mathfrak{gl}}}
\def\GL{{\mathrm{GL}}}

\def\root{\mathrm{root}}
\def\and{\mathrm{and}}

\def\diag{\mathrm{diag}}

\def\ad{{\mathrm{ad}}}

\def\half{{\frac{1}{2}}}
\def\gvmhom{$M_\liep(\mu)\to M_\liep(\lambda)$ }

\def\mmmm{\hspace{-15pt}}

%% file: dynkin2bila.tex
\let\ssize\scriptstyle
\newif\ifFIRST\newdimen\MAXright\MAXright0pt
\def\sdynkin{\bgroup\eightpoint\dynkin}
\def\endsdynkin{\enddynkin\egroup}
\def\dynkin{\bgroup\FIRSTtrue\hskip.5em\setbox1\hbox{$\diagup$}%
\setbox2\hbox{$\diagdown$}%
\setbox0\hbox to2\wd1{\hrulefill}%
\setbox3\hbox{$\bullet$}%
\setbox4\hbox{$\times$}%
\setbox7\hbox{$\circ$}
\def\whiteroot##1{\ifFIRST\setbox5\hbox{$##1$}\ifdim\wd5>1.3em
\hskip-.5em\hskip.5\wd5\fi\fi\FIRSTfalse
\hskip-.25em\raise1.5\wd3\hbox to0pt{\hss\hskip.45em$
\ssize##1$\hss}\copy7\hskip-.25em\setbox6\hbox{$##1$}
\MAXright\wd6}
\def\root##1{\ifFIRST\setbox5\hbox{$##1$}\ifdim\wd5>1.3em%
\hskip-.5em\hskip.5\wd5\fi\fi\FIRSTfalse%
\hskip-.25em\raise1.5\wd3\hbox to0pt{\hss\hskip.45em$%
\ssize##1$\hss}\copy3\hskip-.25em\setbox6\hbox{$##1$}%
\MAXright\wd6}%
\def\whitedroot##1{\ifFIRST\setbox5\hbox{$##1$}\ifdim\wd5>1.3em
\hskip-.5em\hskip.5\wd5\fi\fi\FIRSTfalse
\hskip-.25em\lower1.8\wd3\hbox to0pt{\hss\hskip.45em$
\ssize##1$\hss}\copy7\hskip-.25em\setbox6\hbox{$##1$}
\MAXright\wd6}%
\def\whiterroot##1{\hskip-.25em\copy7\hbox to0pt{\hskip.3em$\ssize##1$\hss}%
\hskip-.25em\setbox6\hbox{\hskip.6em$##1##1$}%
\MAXright\wd6}%
\def\droot##1{\ifFIRST\setbox5\hbox{$##1$}\ifdim\wd5>1.3em%
\hskip-.5em\hskip.5\wd5\fi\fi\FIRSTfalse%
\hskip-.25em\lower1.8\wd3\hbox to0pt{\hss\hskip.45em$%
\ssize##1$\hss}\copy3\hskip-.25em\setbox6\hbox{$##1$}%
\MAXright\wd6}%
\def\rroot##1{\hskip-.25em\copy3\hbox to0pt{\hskip.3em$\ssize##1$\hss}%
\hskip-.25em\setbox6\hbox{\hskip.6em$##1##1$}%
\MAXright\wd6}%
\def\norroot##1{\hskip-.36em\copy4\hbox to0pt{\hskip.3em$\ssize##1$\hss}%
\hskip-.48em\setbox6\hbox{\hskip.6em$##1##1$}%
\MAXright\wd6}%
\def\noroot##1{\ifFIRST\setbox5\hbox{$##1$}\ifdim\wd5>1.3em%
\hskip-.5em\hskip.5\wd5\fi\fi\FIRSTfalse%
\hskip-.36em\raise1.5\wd3\hbox to0pt{\hss\hskip.6em$%
\ssize##1$\hss}\copy4\hskip-.38em\setbox6\hbox{$##1$}%
\MAXright\wd6}%
\def\nodroot##1{\ifFIRST\setbox5\hbox{$##1$}\ifdim\wd5>1.3em%
\hskip-.5em\hskip.5\wd5\fi\fi\FIRSTfalse%
\hskip-.36em\lower1.8\wd3\hbox to0pt{\hss\hskip.6em$%
\ssize##1$\hss}\copy4\hskip-.38em\setbox6\hbox{$##1$}%
\MAXright\wd6}%
\def\nolink{\hskip\wd0}
\def\link{\raise.22em\copy0}%
\def\llink##1{\raise.32em\copy0\hskip-\wd0%
\raise.12em\copy0\hskip-.5\wd0\hbox to0pt{\hss$##1$\hss}\hskip.5\wd0}%
\def\lllink##1{\raise.22em\copy0\hskip-\wd0\raise.32em\copy0\hskip-\wd0%
\raise.12em\copy0\hskip-.5\wd0\hbox to0pt{\hss$##1$\hss}\hskip.5\wd0}%
\def\rootupright##1{\hbox to0pt{\raise.45em\copy1\hskip-.25em\raise1.3\ht1%
\hbox{\copy3\hskip.3em$\ssize##1$}\hss}%
\setbox6\hbox{\hskip.6em\copy1\copy1$##1##1$}%
\ifdim\MAXright<\wd6\MAXright\wd6\fi}%
\def\whiterootupright##1{\hbox to0pt{\raise.45em\copy1\hskip-.25em\raise1.3\ht1
\hbox{\copy7\hskip.3em$\ssize##1$}\hss}
\setbox6\hbox{\hskip.6em\copy1\copy1$##1##1$}
\ifdim\MAXright<\wd6\MAXright\wd6\fi}
\def\norootupright##1{\hbox to0pt{\raise.45em\copy1\hskip-.36em\raise1.3\ht1%
\hbox{\copy4\hskip.3em$\ssize##1$}\hss}%
\setbox6\hbox{\hskip.6em\copy1\copy1$##1##1$}%
\ifdim\MAXright<\wd6\MAXright\wd6\fi}%
\def\rootdownright##1{\hbox to0pt{\raise-.5em\copy2\hskip-.25em\raise-1.35\ht1%
\hbox{\copy3\hskip.3em$\ssize##1$}\hss}\setbox6%
\hbox{\hskip.6em\copy2\copy2$##1##1$}%
\ifdim\MAXright<\wd6\MAXright\wd6\fi}%
\def\whiterootdownright##1{\hbox to0pt{\raise-.5em\copy2\hskip-.25em\raise-1.35\ht1
\hbox{\copy7\hskip.3em$\ssize##1$}\hss}\setbox6
\hbox{\hskip.6em\copy2\copy2$##1##1$}
\ifdim\MAXright<\wd6\MAXright\wd6\fi}
\def\rootdown##1{\hbox to0pt{\hskip-.05em\vrule height.25em depth.65em%
\hskip-.25em\raise-.95em\hbox{\copy3\hskip.3em$\ssize##1$}\hss}%
\setbox6\hbox{$##1$}%
\ifdim\MAXright<\wd6\MAXright\wd6\fi}%
\def\whiterootdown##1{\hbox to0pt{\hskip-.05em\vrule height.25em depth.65em
\hskip-.25em\raise-.95em\hbox{\copy7\hskip.3em$\ssize##1$}\hss}
\setbox6\hbox{$##1$}
\ifdim\MAXright<\wd6\MAXright\wd6\fi}
\def\dots{\hskip.5em\cdots\hskip.5em}}%
\def\enddynkin{\ifdim\MAXright>1em\hskip.5\MAXright\else\hskip.5em\fi\egroup}%
 
This macro is used for creation of Dynkin diagrams in AMSTeX.
The usage:
In math-mode, between \dynkin and \enddynkin,
\root#1 makes a node with label #1 over this node,
\rroot#1 and \droot do the same with labels on right and down,
\noroot and \norroot do the same with cross instead of bullet,
\link links two adjacent \root's, \llink#1 links two adjacent nodes
with double-line with #1 (probably '>' or '<') in the middle, \lllink is does
same with three lines, \dots does \cdots between the two adjacent nodes,
\rootdown#1 makes a linked node down labeled by #1, \rootupright#1 and
\rootdownright#1 are labeled linked nodes in the indicated directions 
which do not count for the jorizontal dimensions.

\sdynkin ... \endsdynkin works exactly as above but in smaller
version (eightpoint)

Examples:
$\dynkin \root{a_1}\link\root{a_2}\dots\root{a_{n-1}}\link\root{a_n}
\enddynkin$
or
$\dynkin \root{}\lllink>\root{}\enddynkin$
or
$\dynkin \root{a}\link\root{b}\rootupright{c}\rootdownright{d}\enddynkin$
or
$\dynkin \root{}\link\root{}\rootupright{}\link\root{}\enddynkin$.
or (white roots!)
{$\dynkin \whiteroot{0}\link\noroot{0}\dots\root{0}\link\whiteroot{0}\rootupright{0}\whiterootdownright{0}\enddynkin$}. 